\documentclass[11pt]{amsart}

\usepackage[active]{srcltx}
\usepackage{amsfonts}
\usepackage{amssymb}

\title{A note on representations of some affine vertex algebras of type $D$}
\author{Ozren Per\v{s}e}
\date{}
\begin{document}
\def \Z{\Bbb Z}
\def \C{\Bbb C}
\def \R{\Bbb R}
\def \Q{\Bbb Q}
\def \N{\Bbb N}
\def \tr{{\rm tr}}
\def \span{{\rm span}}
\def \Res{{\rm Res}}
\def \End{{\rm End}}
\def \E{{\rm End}}
\def \Ind {{\rm Ind}}
\def \Irr {{\rm Irr}}
\def \Aut{{\rm Aut}}
\def \Hom{{\rm Hom}}
\def \mod{{\rm mod}}
\def \ann{{\rm Ann}}
\def \<{\langle}
\def \>{\rangle}
\def \t{\tau }
\def \a{\alpha }
\def \e{\epsilon }
\def \l{\lambda }
\def \L{\Lambda }
\def \g{\gamma}
\def \b{\beta }
\def \om{\omega }
\def \o{\omega }
\def \c{\chi}
\def \ch{\chi}
\def \cg{\chi_g}
\def \ag{\alpha_g}
\def \ah{\alpha_h}
\def \ph{\psi_h}
\def \be{\begin{equation}\label}
\def \ee{\end{equation}}
\def \bl{\begin{lem}\label}
\def \el{\end{lem}}
\def \bt{\begin{thm}\label}
\def \et{\end{thm}}
\def \bp{\begin{prop}\label}
\def \ep{\end{prop}}
\def \br{\begin{rem}\label}
\def \er{\end{rem}}
\def \bc{\begin{coro}\label}
\def \ec{\end{coro}}
\def \bd{\begin{de}\label}
\def \ed{\end{de}}
\def \pf{{\bf Proof. }}
\def \voa{{vertex operator algebra}}

\newtheorem{thm}{Theorem}[section]
\newtheorem{prop}[thm]{Proposition}
\newtheorem{coro}[thm]{Corollary}
\newtheorem{conj}[thm]{Conjecture}
\newtheorem{lem}[thm]{Lemma}
\newtheorem{rem}[thm]{Remark}
\newtheorem{de}[thm]{Definition}
\newtheorem{hy}[thm]{Hypothesis}
\newtheorem{ex}[thm]{Example}
\makeatletter \@addtoreset{equation}{section}
\def\theequation{\thesection.\arabic{equation}}
\makeatother \makeatletter

\newcommand{\bea}{\begin{eqnarray}}
\newcommand{\eea}{\end{eqnarray}}
    \newcommand{\nno}{\nonumber}
    \newcommand{\lbar}{\bigg\vert}
    \newcommand{\p}{\partial}
    \newcommand{\dps}{\displaystyle}
    \newcommand{\bra}{\langle}
    \newcommand{\ket}{\rangle}
 \newcommand{\res}{\mbox{\rm Res}}
\renewcommand{\hom}{\mbox{\rm Hom}}
  \newcommand{\epf}{\hspace{2em}$\Box$}
 \newcommand{\epfv}{\hspace{1em}$\Box$\vspace{1em}}
\newcommand{\nord}{\mbox{\scriptsize ${\circ\atop\circ}$}}
\newcommand{\wt}{\mbox{\rm wt}\ }

\maketitle
\begin{abstract}
In this note we construct a series of singular vectors in universal
affine vertex operator algebras associated to $D_{\ell}^{(1)}$ of
levels $n-\ell+1$, for $n \in \Z _{>0}$. For $n=1$, we study the
representation theory of the quotient vertex operator algebra modulo
the ideal generated by that singular vector. In the case $\ell =4$,
we show that the adjoint module is the unique irreducible ordinary
module for simple vertex operator algebra $L_{D_{4}}(-2,0)$. We also
show that the maximal ideal in associated universal affine vertex
algebra is generated by three singular vectors.
\end{abstract}


\footnotetext[1]{ {\em 2000 Mathematics Subject Classification.}
Primary 17B69; Secondary 17B67, 81R10.}

\section{Introduction}

The classification of irreducible modules for simple vertex operator
algebra $L_{\frak g}(k, 0)$ associated to affine Lie algebra
$\hat{\frak g}$ of level $k$ is still an open problem for general $k
\in \mathbb{C}$ ($k \neq - h^{\vee}$). This problem is connected
with the description of the maximal ideal in the universal affine
vertex algebra $N_{\frak g}(k, 0)$. One approach to this
classification problem is through construction of singular vectors
in $N_{\frak g}(k, 0)$.

The known (non-generic) cases include positive integer levels (cf.
\cite{FZ}, \cite{L}, \cite{MP}) and some special cases of rational
admissible levels, in the sense of Kac and Wakimoto \cite{KW} (cf.
\cite{A1}, \cite{A-jpaa}, \cite{AM}, \cite{AL}, \cite{DLM},
\cite{P1}, \cite{P-glas}). It turns out that negative integer levels
also have some interesting properties. They appeared in bosonic
realizations in \cite{FF}, and also recently in the context of
conformal embeddings (cf. \cite{AP-ART}).

In this note we study a vertex operator algebra associated to affine
Lie algebra of type $D_{\ell}^{(1)}$ and negative integer level
$-\ell+2$. This level appeared in \cite{AP-ART} in the context of
conformal embedding of $L_{B_{\ell-1}}(-\ell+2, 0)$ into
$L_{D_{\ell}}(-\ell+2,0)$. This conformal embedding is in some sense
similar to the conformal embedding of
$L_{D_{\ell}}(-\ell+\frac{3}{2},0)$ into
$L_{B_{\ell}}(-\ell+\frac{3}{2},0)$.

We will show that there are also similarities in singular vectors in
universal affine vertex algebras $N_{B_{\ell}}(-\ell+\frac{3}{2},0)$
(studied in \cite{P1}) and $N_{D_{\ell}}(-\ell+2,0)$. More
generally, we construct a series of singular vectors
$$v_n = \Big(\sum _{i=2}^{\ell} e_{\epsilon_1 - \epsilon_i}(-1)
e_{\epsilon_1 + \epsilon_i}(-1)\Big) ^n {\bf 1} $$
in $N_{D_{\ell}}(n-\ell+1,0)$, for any $n \in \Z _{>0}$. For $n=1$,
we study the representation theory of the quotient
$N_{D_{\ell}}(-\ell+2,0)$ modulo the ideal generated by $v_1$. Using
the methods from \cite{A1}, \cite{A2}, \cite{AM}, \cite{MP}, we
obtain the classification of irreducible weak modules in the
category $\mathcal{O}$ for that vertex algebra. It turns out that
there are infinitely many of these modules.

In the special case $\ell =4$, we obtain the classification of
irreducible weak modules from the category $\mathcal{O}$ for simple
vertex operator algebra $L_{D_4}(-2,0)$. This vertex algebra also
appeared in \cite{AP-ART} in the context of conformal embedding of
$L_{G_2}(-2,0)$ into $L_{D_4}(-2,0)$. It follows that there are
finitely many irreducible weak $L_{D_4}(-2,0)$--modules from the
category $\mathcal{O}$, that the adjoint module is the unique
irreducible ordinary $L_{D_4}(-2,0)$--module, and that every
ordinary $L_{D_4}(-2,0)$--module is completely reducible. We also
show that the maximal ideal in $N_{D_{4}}(-2,0)$ is generated by
three singular vectors.

The author thanks Dra\v{z}en Adamovi\'{c} for his helpful advice and
valuable suggestions.

\section{Preliminaries}

We assume that the reader is familiar with the notion of vertex
operator algebra (cf. \cite{Bo}, \cite{FHL}, \cite{FLM}, \cite{FrB},
\cite{FZ}, \cite{K2}, \cite{L}, \cite{LL}) and Kac-Moody algebra
(cf. \cite{K1}).

Let $V$ be a vertex operator algebra. Denote by $A(V)$ the
associative algebra introduced in \cite{Z}, called the Zhu's algebra
of $V$. As a vector space, $A(V)$ is a quotient of $V$, and we
denote by $[a]$ the image of $a \in V$ under the projection of $V$
onto $A(V)$. We recall the following fundamental result from
\cite{Z}:

\begin{prop}
The equivalence classes of the irreducible $A(V)$--modules and the
equivalence classes of the irreducible ${\Z}_{+}$--graded weak
$V$--modules are in one-to-one correspondence.
\end{prop}

Let ${\frak g}$ be a simple Lie algebra with a triangular
decomposition  ${\frak g}={\frak n}_{-} \oplus {\frak h} \oplus
{\frak n}_{+}$, and $\hat{\frak g}$ the (untwisted) affine Lie
algebra associated to ${\frak g}$. Denote by $V(\mu)$ the
irreducible highest weight ${\frak g}$--module with highest weight
$\mu$, and by $L(k, \mu)$ the irreducible highest weight $\hat{\frak
g}$--module with highest weight $k \Lambda _0 +\mu$.

Furthermore, denote by $N(k, 0)$ (or $N_{\frak g}(k, 0)$) the
universal affine vertex algebra associated to $\hat{\frak g}$ of
level $k \in \mathbb{C}$. For $k\ne - h^{\vee}$, $N(k,0)$ is a
vertex operator algebra with Segal-Sugawara conformal vector, and
$L(k,0)$ is a simple vertex operator algebra. The Zhu's algebra of
$N(k, 0)$ was determined in \cite{FZ}:

\begin{prop}
The associative algebra $A(N(k,0))$ is canonically isomorphic to $U
(\frak g ) $. The isomorphism is given by $F:A(N(k,0)) \to U (\frak
g )$
\[
F([x_1(-n_1 -1)\cdots x_m(-n_m -1){\bf 1}])= (-1)^{n_1+\cdots +n_m}
x_m \cdots x_1,
\]
for any $x_1, \ldots ,x_m \in \frak g$ and any $n_1, \ldots ,n_m \in
{\Z}_{+}$.
\end{prop}

We have:

\begin{prop} \label{prop-zhu-quot}
Assume that a $\hat{\frak g}$--submodule $J$ of $N(k,0)$ is
generated by $m$ singular vectors ($m \in \Z _{>0}$), i.e.
$J=U(\hat{\frak g})v^{(1)}+ \ldots + U(\hat{\frak g})v^{(m)}$. Then
\[
A(N(k,0)/J) \cong U(\frak g)/I,
\]
where $I$ is the two-sided ideal of $U(\frak g)$ generated by
$u^{(1)}=F([v^{(1)}]), \ldots , u^{(m)}=F([v^{(m)}])$.
\end{prop}

Let $J=U(\hat{\frak g})v^{(1)}+ \ldots + U(\hat{\frak g})v^{(m)}$ be
a $\hat{\frak g}$--submodule of $N(k,0)$ generated by singular
vectors $v^{(1)}, \ldots , v^{(m)}$. Now we recall the method from
\cite{A1}, \cite{A2}, \cite{AM}, \cite{MP} for the classification of
irreducible $A(N(k,0)/J)$--modules from the category $\mathcal{O}$
by solving certain systems of polynomial equations.

Denote by $_L$ the adjoint action of  $U(\frak g)$ on $U(\frak g)$
defined by $ X_Lf=[X,f]$ for $X \in \frak g$ and $f \in U(\frak g)$.
Let $R^{(j)}$ be a $U(\frak g)$--submodule of $U(\frak g)$ generated
by the vector $u^{(j)}=F([v^{(j)}])$ under the adjoint action, for
$j=1, \ldots ,m$. Clearly, $R^{(j)}$ is an irreducible highest
weight $U(\frak g)$--module. Let $R^{(j)} _{0}$ be the zero-weight
subspace of $R^{(j)}$.

The next proposition follows from \cite{A1}, \cite{AM}, \cite{MP}:

\begin{prop} \label{prop-R0}
Let $V(\mu)$ be an irreducible highest weight $U(\frak g)$--module
with the highest weight vector $v_{\mu}$, for $\mu \in {\frak
h}^{*}$. The following statements are equivalent:
\item[(1)] $V(\mu)$ is an $A(N(k,0)/J)$--module,
\item[(2)] $R^{(j)} V(\mu)=0$, for every $j=1, \ldots ,m$,
\item[(3)] $R^{(j)} _{0}v_{\mu}=0$, for every $j=1, \ldots ,m$.
\end{prop}

Let $r \in R^{(j)} _{0}$. Clearly there exists the unique polynomial
$p_{r} \in S( \frak h)$ such that
\[
rv_{\mu}=p_{r}(\mu)v_{\mu}.
\]

Set $ {\mathcal P}^{(j)} _{0}=\{ \ p_{r} \ \vert \ r \in R
^{(j)}_{0} \}$, for $j=1, \ldots ,m$. We have:

\begin{coro} \label{cor-class-polynom} There is one-to-one correspondence between
\item[(1)] irreducible $A(N(k,0)/J)$--modules
from the category $\mathcal{O}$,
\item[(2)] weights $\mu \in {\frak h}^{*}$ such that
$p(\mu)=0$ for all $p \in {\mathcal P}^{(j)}_{0}$, for every $j=1,
\ldots ,m$.
\end{coro}

In the case $m=1$, we use the notation $R$, $R_0$ and ${\mathcal
P}_{0}$ for $R^{(1)}$, $R^{(1)} _0$ and ${\mathcal P}^{(1)}_{0}$,
respectively.

\section{Vertex operator algebra associated to $D_{\ell}^{(1)}$ of level $-\ell+2$}
\label{sect-class}

In this section we study the representation theory of the quotient
of universal affine vertex operator algebra associated to
$D_{\ell}^{(1)}$ of level $-\ell+2$, modulo the ideal generated by a
singular vector of conformal weight two.

Denote by ${\frak g}$ the simple Lie algebra of type $D_{\ell}$. We
fix the root vectors for ${\frak g}$ as in \cite{Bou}, \cite{FF}. We
have:

\begin{thm} \label{sing-vectD}
Vector
$$v_n = \Big(\sum _{i=2}^{\ell} e_{\epsilon_1 - \epsilon_i}(-1)
e_{\epsilon_1 + \epsilon_i}(-1)\Big) ^n {\bf 1} $$ is a singular
vector in $N_{D_{\ell}}(n-\ell+1,0)$, for any $n \in \Z _{>0}$.
\end{thm}
\begin{proof} Direct verification of relations $e_{\epsilon_k -
\epsilon_{k+1}}(0)v_n=0$, for $k=1,\ldots ,\ell-1$,
$e_{\epsilon_{\ell-1} + \epsilon_{\ell}}(0)v_n=0$ and $f_{\epsilon_1
+ \epsilon_2}(1)v_n=0$.
\end{proof}

In the case $n=1$, we obtain the singular vector
\bea \label{sing-D-lowestlevel}
 v = \sum _{i=2}^{\ell} e_{\epsilon_1 - \epsilon_i}(-1)
e_{\epsilon_1 + \epsilon_i}(-1) {\bf 1} \eea
in $N_{D_{\ell}}(-\ell+2,0)$.

\begin{rem} Vector $v$ from relation (\ref{sing-D-lowestlevel}) has a similar formula as
singular vector
$$ -\frac{1}{4} e_{\epsilon_1}(-1) ^2 {\bf 1} +  \sum _{i=2}^{\ell} e_{\epsilon_1 - \epsilon_i}(-1)
e_{\epsilon_1 + \epsilon_i}(-1) {\bf 1} $$
for $B_{\ell}^{(1)}$ in $N_{B_{\ell}}(-\ell+\frac{3}{2},0)$ . The
representation theory of the quotient of
$N_{B_{\ell}}(-\ell+\frac{3}{2},0)$ modulo the ideal generated by
that vector was studied in \cite{P1}.
\end{rem}

We will consider representations of the vertex operator algebra
$$\mathcal{V}_{D_{\ell}}(-\ell+2,0)= \frac{N_{D_{\ell}}(-\ell+2,0)}
{U(\hat{\frak g})v }.$$

Proposition \ref{prop-zhu-quot} gives:

\begin{prop}
The associative algebra $A(\mathcal{V}_{D_{\ell}}(-\ell+2,0))$ is
isomorphic to the algebra $U({\frak g})/I$, where $I$ is the
two-sided ideal of $U({\frak g})$ generated by
\bea u = \sum _{i=2}^{\ell} e_{\epsilon_1 - \epsilon_i}
e_{\epsilon_1 + \epsilon_i}. \nonumber \eea
\end{prop}

We have the following classification:

\begin{thm} \label{thm-class-katO}
For any subset $S=\{ i_{1}, \ldots ,i_{k} \} \subseteq \{1,2, \ldots
, \ell-2 \}$, \linebreak $i_{1}< \ldots <i_{k}$, and $t \in
\mathbb{C}$, we define weights
\begin{eqnarray*}
&& \mu _{S,t}= \sum _{j=1}^{k}\left(i_{j}+2 \sum _{s=j+1}^{k}
(-1)^{s-j}i_{s} + (-1)^{k-j+1}(t+ \ell- 1)\right) \omega _{i_{j}} +t \omega _{\ell -1}, \\
&& \mu _{S,t}'=\sum _{j=1}^{k}\left(i_{j}+2 \sum _{s=j+1}^{k}
(-1)^{s-j}i_{s} + (-1)^{k-j+1}(t+\ell-1)\right) \omega _{i_{j}}
+t\omega _{\ell},
\end{eqnarray*}
where $\omega _{1}, \ldots , \omega _{\ell}$ are fundamental weights
for ${\frak g}$. Then the set
$$ \{ L_{D_{\ell}}(-\ell+2, \mu _{S,t}), L_{D_{\ell}}(-\ell+2, \mu _{S,t}') \ \vert \ S \subseteq \{1,2, \ldots , \ell-2
\}, t \in \mathbb{C} \}$$ provides the complete list of irreducible
weak $\mathcal{V}_{D_{\ell}}(-\ell+2,0)$--modules from the category
$\mathcal{O}$.
\end{thm}
\begin{proof} We use the method for classification of irreducible
$A({\mathcal V}_{D_{\ell}}(-\ell+2,0))$--modules in the category
$\mathcal{O}$ from Corollary \ref{cor-class-polynom}. In this case
$R \cong V_{D_{\ell}}(2 \omega _1)$, and similarly as in \cite[Lemma
28]{P1} one obtains that
$$ \dim R_0 = \ell -1.$$
Furthermore, one obtains by direct calculation that
\bea
&& (f_{\epsilon_1 - \epsilon_2}f_{\epsilon_1 + \epsilon_2}) _L u
\in p_1(h) + U(\frak g){\frak n}_{+}, \nonumber \\
 && (f_{\epsilon_1 - \epsilon_{i+1}}f_{\epsilon_1 +
\epsilon_{i+1}} - f_{\epsilon_1 - \epsilon_{i}}f_{\epsilon_1 +
\epsilon_{i}}) _L u \in p_{i}(h) + U(\frak g){\frak n}_{+}, \ i=2,
\ldots ,\ell -1, \nonumber \eea
where
 \bea \label{rel-pol1} p_{i}(h)=h_{i}(h_{\epsilon_i + \epsilon_{i+1}}+l-i-1),
\quad \mbox{for } i=1, \ldots ,\ell -1 \eea
are linearly independent polynomials in ${\mathcal P}_{0}$. Here
$h_{i}$ ($i=1, \ldots ,\ell$) denote the simple coroots for ${\frak
g}$ and
$$ h_{\epsilon_i + \epsilon_{i+1}}=h_{i}+2h_{i+1}+ \ldots
+2h_{\ell-2}+h_{\ell -1} + h_{\ell}, \quad \mbox{for } i < \ell -1.
$$
Corollary \ref{cor-class-polynom} now implies that the highest
weights of irreducible $A({\mathcal
V}_{D_{\ell}}(-\ell+2,0))$--modules from the category $\mathcal{O}$
are given as solutions of polynomial equations
\bea \label{rel-pol2} p_{i}(h)=0, \ i=1, \ldots ,\ell -1. \eea

First we note that for $i=\ell -1$, we obtain the equation
$$h_{\ell -1}h_{\ell}=0.$$
Thus, either $h_{\ell -1}=0$ or $h_{\ell}=0$. Assume first that
$h_{\ell -1}=0$, and let $S=\{ i_{1}, \ldots ,i_{k} \}$, $i_{1}<
\ldots <i_{k}$ be the subset of $\{1,2, \ldots , \ell-2 \}$ such
that $h_{i}=0$ for $i \notin S$ and $h_{i} \neq 0$ for $i \in S$.
Then we have the system
\begin{eqnarray}
&& h_{i_{1}}+2h_{i_{2}}+ \ldots +2h_{i_{k}}+ h_{\ell} +\ell-i_{1}-1=0, \nno \\
&& h_{i_{2}}+2h_{i_{3}}+ \ldots +2h_{i_{k}}+ h_{\ell}+\ell-i_{2}-1=0, \nno \\
&& \qquad \qquad \qquad \vdots \label{2.6.1.1}\\
&& h_{i_{k-1}}+2h_{i_{k}}+ h_{\ell} +\ell-i_{k-1}-1=0, \nno \\
&& h_{i_{k}}+ h_{\ell}+ \ell-i_{k}-1=0. \nno
\end{eqnarray}
The solution of this system is given by
\begin{eqnarray*}
&& h_{i_{j}}= i_{j}+2 \sum _{s=j+1}^{k} (-1)^{s-j}i_{s} +
(-1)^{k-j+1}(t + \ell- 1), \ \mbox{ for } j=1, \ldots ,k; \\
&& h_{\ell}=t   \quad (t \in \mathbb{C}).
\end{eqnarray*}
It follows that $V_{D_{\ell}}(\mu _{S,t}')$ is an irreducible
$A({\mathcal V}_{D_{\ell}}(-\ell+2,0))$--module. Similarly, the case
$h_{\ell}=0$ gives that $V_{D_{\ell}}(\mu _{S,t})$ is irreducible
$A({\mathcal V}_{D_{\ell}}(-\ell+2,0))$--module. We conclude that
the set
$$ \{ V_{D_{\ell}}(\mu _{S,t}), V_{D_{\ell}}(\mu _{S,t}') \ \vert \ S \subseteq \{1,2, \ldots , \ell-2
\}, t \in \mathbb{C} \}$$ provides the complete list of irreducible
$A(\mathcal{V}_{D_{\ell}}(-\ell+2,0))$-modules from the category
$\mathcal{O}$. The claim of theorem now follows from Zhu's theory.
\end{proof}

\begin{ex} For $\ell =4$, we have subsets $S = \emptyset, \{ 1 \}, \{ 2 \}, \{ 1, 2
\}$ of the set $\{ 1, 2 \}$, so we obtain that the set
\bea
&& \{ L_{D_{\ell}}(-\ell+2, t \omega _3), L_{D_{\ell}}(-\ell+2, t
\omega _4), L_{D_{\ell}}(-\ell+2, (-2-t) \omega _1 + t \omega _3),
\nonumber \\
&& L_{D_{\ell}}(-\ell+2, (-2-t) \omega _1 + t \omega _4),
L_{D_{\ell}}(-\ell+2, (-1-t) \omega _2 + t \omega _3), \nonumber \\
&& L_{D_{\ell}}(-\ell+2, (-1-t) \omega _2 + t \omega _4),
L_{D_{\ell}}(-\ell+2, t \omega _1 + (-1-t) \omega _2 + t \omega _3), \nonumber \\
&& L_{D_{\ell}}(-\ell+2, t \omega _1 + (-1-t) \omega _2 + t \omega
_4)  \ \vert \ t \in \mathbb{C} \}  \label{ex-class-D4} \eea
provides the complete list of irreducible weak
$\mathcal{V}_{D_4}(-2,0)$--modules from the category $\mathcal{O}$.
\end{ex}

Recall that a module for vertex operator algebra is called ordinary
if $L(0)$ acts semisimply with finite-dimensional weight spaces. We
have:
\begin{coro} \label{cor-class}
 The set
\begin{eqnarray*}
\left\{ L_{D_{\ell}}(-\ell+2, t \omega _{\ell -1}),
L_{D_{\ell}}(-\ell+2, t \omega _{\ell}) \ \vert \ t \in \Z _{\geq 0}
\right\}
\end{eqnarray*}
provides the complete list of irreducible ordinary ${\mathcal
V}_{D_{\ell}}(-\ell+2,0)$--modules.
\end{coro}
\begin{proof}
If $L_{D_{\ell}}(-\ell+2, \mu)$ is an ordinary ${\mathcal
V}_{D_{\ell}}(-\ell+2,0)$--module, then $\mu$ is a dominant integral
weight. Then $\mu (h_{\epsilon_i + \epsilon_{i+1}}) \in \Z _{\geq
0}$, for $i=1, \ldots ,\ell -1$. Relations (\ref{rel-pol1}) and
(\ref{rel-pol2}) then give that
$$\mu (h_i)=0, \quad \mbox{for } i=1, \ldots ,\ell -2,$$
and $\mu (h_{\ell -1})=0$ or $\mu (h_{\ell})=0$. Thus, $\mu =t
\omega _{\ell -1}$ or $\mu =t \omega _{\ell}$, and $t \in \Z _{\geq
0}$ since $\mu$ is a dominant integral weight.
\end{proof}

It follows that:

\begin{coro} \label{cor-class-simple}
 The set of irreducible ordinary $L_{D_{\ell}}(-\ell+2,0)$--modules
 is a subset of the set
\begin{eqnarray*}
\left\{ L_{D_{\ell}}(-\ell+2, t \omega _{\ell -1}),
L_{D_{\ell}}(-\ell+2, t \omega _{\ell}) \ \vert \ t \in \Z _{\geq 0}
\right\}.
\end{eqnarray*}
\end{coro}

\section{Case $\ell =4$}

In this section we study the case $\ell =4$. We determine the
classification of irreducible weak $L_{D_{4}}(-2,0)$--modules from
the category $\mathcal{O}$. It turns out that there are finitely
many of these modules and that the adjoint module is the unique
irreducible ordinary $L_{D_{4}}(-2,0)$--module. We also show that
the maximal ideal in $N_{D_{4}}(-2,0)$ is generated by three
singular vectors.

Denote by $\theta$ the automorphism of $N_{D_{4}}(-2,0)$ induced by
the automorphism of the Dynkin diagram of $D_4$ of order three such
that
\bea
\theta (\epsilon_1 - \epsilon_2) = \epsilon_3 - \epsilon_4, \ \theta
(\epsilon_2 - \epsilon_3) = \epsilon_2 - \epsilon_3, \ \theta
(\epsilon_3 - \epsilon_4) = \epsilon_3 + \epsilon_4, \ \theta
(\epsilon_3 + \epsilon_4) = \epsilon_1 - \epsilon_2. \nonumber \eea
Relation (\ref{sing-D-lowestlevel}) implies that
$$v= (e_{\epsilon_1 - \epsilon_2}(-1)
e_{\epsilon_1 + \epsilon_2}(-1)+  e_{\epsilon_1 - \epsilon_3}(-1)
e_{\epsilon_1 + \epsilon_3}(-1) + e_{\epsilon_1 - \epsilon_4}(-1)
e_{\epsilon_1 + \epsilon_4}(-1) ) {\bf 1}$$
is a singular vector in $N_{D_{4}}(-2,0)$. Furthermore,
$$\theta (v) = (e_{\epsilon_3 - \epsilon_4}(-1)
e_{\epsilon_1 + \epsilon_2}(-1)-  e_{\epsilon_2 - \epsilon_4}(-1)
e_{\epsilon_1 + \epsilon_3}(-1) + e_{\epsilon_2 + \epsilon_3}(-1)
e_{\epsilon_1 - \epsilon_4}(-1) ) {\bf 1},$$
and
$$ \theta ^2 (v)  = (e_{\epsilon_3 + \epsilon_4}(-1)
e_{\epsilon_1 + \epsilon_2}(-1)-  e_{\epsilon_2 + \epsilon_4}(-1)
e_{\epsilon_1 + \epsilon_3}(-1) + e_{\epsilon_1 + \epsilon_4}(-1)
e_{\epsilon_2 + \epsilon_3}(-1) ) {\bf 1}$$
are also singular vectors in $N_{D_{4}}(-2,0)$. We consider the
vertex operator algebra
$${\widetilde L}_{D_{4}}(-2,0) = \frac{N_{D_4}(-2,0)}{J}, $$
where $J$ is the ideal in $N_{D_{4}}(-2,0)$ generated by vectors
$v$, $\theta (v)$ and $\theta ^2 (v)$.

Proposition \ref{prop-zhu-quot} gives that the associative algebra
$A({\widetilde L}_{D_{4}}(-2,0))$ is isomorphic to the algebra
$U({\frak g})/I$, where $I$ is the two-sided ideal of $U({\frak g})$
generated by $u$, $\theta (u)$ and $\theta ^2(u)$, and
$$u= e_{\epsilon_1 - \epsilon_2}
e_{\epsilon_1 + \epsilon_2}+  e_{\epsilon_1 - \epsilon_3}
e_{\epsilon_1 + \epsilon_3} + e_{\epsilon_1 - \epsilon_4}
e_{\epsilon_1 + \epsilon_4}.$$

\begin{prop} \label{prop-class-tilda} We have:
\item[(i)]
The set
\bea  \{ L_{D_4}(-2, 0), L_{D_4}(-2, -2 \omega _1), L_{D_4}(-2, -2
\omega _3), L_{D_4}(-2, -2 \omega _4), L_{D_4}(-2, - \omega _2)
\}.\nonumber \eea
 provides a complete list of irreducible weak
${\widetilde L}_{D_{4}}(-2,0)$--modules from the category
$\mathcal{O}$.
\item[(ii)]
$L_{D_{4}}(-2,0)$ is the unique irreducible ordinary module for
${\widetilde L}_{D_{4}}(-2,0)$.
\end{prop}
\begin{proof} (i) We use the method for classification from Corollary \ref{cor-class-polynom}.
In this case $R ^{(1)} \cong V_{D_4}(2 \omega _1)$, $R ^{(2)} \cong
V_{D_4}(2 \omega _3)$, $R ^{(3)} \cong V_{D_4}(2 \omega _4)$ and
$$ \dim R ^{(1)} _0 = \dim R ^{(2)} _0= \dim R ^{(3)} _0= 3.$$
Using polynomials from relation (\ref{rel-pol1}) and automorphisms
$\theta$ and $\theta ^2$, one obtains that the highest weights $\mu$
of $A({\widetilde L}_{D_{4}}(-2,0))$--modules $V_{D_4}(\mu)$ are
obtained as solutions of these $9$ polynomial equations:
\bea && h_{\epsilon_1 - \epsilon_2}(h_{\epsilon_1 + \epsilon_2}
+2)=0  \nonumber \\
&& h_{\epsilon_2 - \epsilon_3}(h_{\epsilon_2 + \epsilon_3}
+1)=0  \nonumber \\
&& h_{\epsilon_3 - \epsilon_4}h_{\epsilon_3 + \epsilon_4} =0
\nonumber \\
&& h_{\epsilon_3 - \epsilon_4}(h_{\epsilon_1 + \epsilon_2}
+2)=0  \nonumber \\
&& h_{\epsilon_2 - \epsilon_3}(h_{\epsilon_1 + \epsilon_4}
+1)=0  \nonumber \\
&& h_{\epsilon_3 + \epsilon_4}h_{\epsilon_1 - \epsilon_2}
=0  \nonumber \\
&& h_{\epsilon_3 + \epsilon_4}(h_{\epsilon_1 + \epsilon_2}
+2)=0  \nonumber \\
&& h_{\epsilon_2 - \epsilon_3}(h_{\epsilon_1 - \epsilon_4}
+1)=0  \nonumber \\
&& h_{\epsilon_1 - \epsilon_2}h_{\epsilon_3 - \epsilon_4} =0.
\nonumber
 \eea
This easily gives that $\mu=0$, $-2 \omega _1$, $-2 \omega _3$, $-2
\omega _4$ or $- \omega _2$, and the claim follows from Zhu's
theory.

Claim (ii) follows from the fact that $\mu =0$ is the only dominant
integral weight such that $L_{D_4}(-2, \mu)$ is in the set given in
the claim (i).
\end{proof}
We have:
\begin{thm} \label{thm-D4-simple}
Vertex operator algebra ${\widetilde L}_{D_{4}}(-2,0)$ is simple,
i.e.
$$L_{D_{4}}(-2,0) = \frac{N_{D_4}(-2,0)}{U(\hat{\frak g}).v + U(\hat{\frak g}).\theta (v)
+ U(\hat{\frak g}).\theta ^2 (v)}.$$
\end{thm}
\begin{proof}
Let $w$ be a singular vector for $\hat{\frak g}$ in ${\widetilde
L}_{D_{4}}(-2,0)$. The classification result from Proposition
\ref{prop-class-tilda} (ii) implies that $U(\hat{\frak g}).w$ is a
highest weight $\hat{\frak g}$--module with highest weight $-2
\Lambda _0$ and that $w$ is proportional to ${\bf 1}$. The claim
follows.
\end{proof}
We conclude:
\begin{thm}
\item[(i)]
The set
\bea  \{ L_{D_4}(-2, 0), L_{D_4}(-2, -2 \omega _1), L_{D_4}(-2, -2
\omega _3), L_{D_4}(-2, -2 \omega _4), L_{D_4}(-2, - \omega _2)
\}.\nonumber \eea
 provides a complete list of irreducible weak
$L_{D_{4}}(-2,0)$--modules from the category $\mathcal{O}$.
\item[(ii)]
$L_{D_{4}}(-2,0)$ is the unique irreducible ordinary module for
$L_{D_{4}}(-2,0)$.
\item[(iii)] Every ordinary $L_{D_{4}}(-2,0)$--module is completely
reducible.
\end{thm}
\begin{proof} Proposition \ref{prop-class-tilda} and Theorem
\ref{thm-D4-simple} imply claims (i) and (ii).

(iii) Let $M$ be an ordinary $L_{D_{4}}(-2,0)$--module, and let $w$
be a singular vector for $\hat{\frak g}$ in $M$. The classification
result from (ii) implies that $U(\hat{\frak g}).w$ is a highest
weight $\hat{\frak g}$--module with highest weight $-2 \Lambda _0$.
Claim (ii) also implies that any singular vector in $U(\hat{\frak
g}).w$ has highest weight $-2 \Lambda _0$ and it is proportional to
$w$. Thus, $U(\hat{\frak g}).w$ is an irreducible $\hat{\frak
g}$--module and the claim follows.
\end{proof}

\bibliography{thesis}
\bibliographystyle{plain}

\end{document}